\documentclass[11pt,letterpaper]{amsart}

\usepackage[english]{babel}
\usepackage{tikz}
\usetikzlibrary{positioning,arrows,cd}

\usepackage{amsmath,amssymb,amsthm}

\usepackage{enumerate}

\usepackage{verbatim}
\if{} 
\fi

\newtheorem{theorem}{Theorem}[section]
\newtheorem{thmx}{Theorem}

\newtheorem{corx}[thmx]{Corollary}

\newtheorem{proposition}[theorem]{Proposition}
\newtheorem{lemma}[theorem]{Lemma}
\newtheorem{corollary}[theorem]{Corollary}

\theoremstyle{definition}

\theoremstyle{remark}
\newtheorem{example}[theorem]{Example}				
\newtheorem{remark}[theorem]{Remark}

\newcommand{\Z}{{\mathbb Z}}

\newcommand{\Q}{{\mathbb Q}}

\usepackage[pdfpagelabels]{hyperref}

\usepackage{colortbl}
\DefineNamedColor{named}{RoyalBlue}     {cmyk}{1,0.50,0,0}
\DefineNamedColor{named}{BrickRed}      {cmyk}{0,0.89,0.94,0.28}

\AtBeginDocument{%
   \def\MR#1{}
}

\begin{document}

\title{Realizing additive monoids as mapping degree sets}
\author{C. Costoya}
\address{CITMAga, Departamento de Matem{á}ticas, Universidade de Santiago de Compostela, 15782-Santiago de Compostela, Spain.}
\email{cristina.costoya@usc.es}
\author[V.\ Mu\~{n}oz]{V. Mu\~{n}oz}
\address{Departamento de \'Algebra, Geometr\'{\i}a y Topolog\'{\i}a, Universidad Complutense de Madrid, Ciudad Universitaria, 28040 Madrid, Spain}
\email{vicente.munoz@ucm.es}

\author[B.\ Valverde-Morales]{B. Valverde-Morales}
\address{Departamento de \'Algebra, Geometr{\'\i}a y Topolog{\'\i}a, Universidad de M{\'a}laga, 29071 M{\'a}laga, Spain}
\email{brunoval@uma.es}

\author[A.\ Viruel]{A. Viruel}
\address{Departamento de \'Algebra, Geometr{\'\i}a y Topolog{\'\i}a, Universidad de M{\'a}laga, 29071 M{\'a}laga, Spain}
\email{viruel@uma.es}


\begin{abstract}
We show that mapping degree sets remain unchanged under multiplication by finite subsets of $\mathbb Z$ containing $0$, as well as by sets constructed from additive submonoids of $\mathbb Z$ through finitely many applications of addition and multiplication. It follows that every set arising in this way can be realized as a mapping degree set. This yields a large collection of infinite examples, among them finite unions of arithmetic progressions with initial term $0$. Our results broaden the range of known realizable mapping degree sets and provide further evidence toward a question raised by Neofytidis, Wang, and Wang.
\end{abstract}

\maketitle


\section{Introduction}\label{sec:intro}

Let $M$ and $N$ be two closed oriented connected manifolds of the same dimension. The mapping degree set between $M$ and $N$ is defined as:
 $$
 D(M,N)=\{ d \in \Z\, |\, \exists f\colon M\to N, \, \deg(f)=d\}.
 $$

In \cite{NWW}, C.~Neofytidis, S.~Wang, and Z.~Wang study the problem of characterizing those subsets $A\subset \mathbb Z$ that arise as mapping degree sets of closed oriented connected manifolds. Equivalently, they seek to determine which sets can be written in the form
$
A=D(M,N)
$
for some pair of closed oriented connected manifolds $M$ and $N$ of the same dimension. Since the constant map has degree zero, every such set necessarily contains $0$. 

A cardinality argument shows that not every infinite subset of $\Z$ containing $0$ is a mapping degree set \cite[Theorem 1.3]{NWW}. Indeed, there are uncountably many such subsets, but only countably many mapping degree sets arising from pairs of manifolds. Moreover, Löh and Uschold proved that every mapping degree set is recursively enumerable \cite[Proposition A.1]{LohUsc}, yielding further obstructions to realizability.

In \cite[Problem 1.3]{NWW}, Neofytidis, Wang, and Wang ask whether every finite subset of $\mathbb Z$ containing $0$ occurs as a mapping degree set, and whether the same holds for arithmetic progressions containing $0$ \cite[Problem 1.6]{NWW}. 

The finite case was completely solved in \cite{CMuV2} by the first, second, and fourth authors of the present paper, who proved that every finite subset of $\mathbb Z$ containing $0$ is realizable as a mapping degree set. 

In this paper, we approach \cite[Problem 1.6]{NWW} through additive submonoids of $\mathbb Z$, namely subsets of $\mathbb Z$ containing $0$ and closed under addition. We prove that additive submonoids of $\mathbb Z$, and more generally sets obtained from them by finitely many sums and products, occur as mapping degree sets. In fact, we establish a stronger stability result: the product of an arbitrary mapping degree set with either a finite subset of $\mathbb Z$ containing $0$ or a set of the above form is again a mapping degree set.

Let $\mathcal C_0$ be the class of additive submonoids of $\mathbb Z$. For $n\geq 0$, let $\mathcal C_{n+1}$ be the class of all sets obtained as a sum or a product of two elements of $\bigcup_{j\leq n}\mathcal C_j$. Define
$$
\mathcal C:=\bigcup_{n\geq 0}\mathcal C_n.
$$
Let $\mathcal F$ denote the class of finite subsets of $\mathbb Z$ containing $0$. Our main results are summarized in the following theorem.

\begin{thmx}\label{thm:main}
Let \(A \) be a set in $ \mathcal C$ or in   $ \mathcal F$, and let $M_1$ and $M_2$ be a pair of closed oriented connected manifolds of the same dimension. Then there exist closed oriented connected manifolds $N_1,N_2$ such that:
\begin{enumerate}
\item $D(N_1,N_2)=A.
$
\item 
$D(M_1\times N_1,M_2\times N_2)
=
D(M_1,M_2)\cdot A.$
\end{enumerate}
\end{thmx}
The proof of Theorem~\ref{thm:main} is completed in Section~\ref{sec:plus}: the cases of sets in $\mathcal F$ and $\mathcal C$ are established in Theorems~\ref{thm:proofmainf} and \ref{thm:proofmain}, respectively. As a consequence, we obtain the following stability property. 

\begin{corx}
Let $\mathcal R$ be the class of realizable subsets of $\mathbb Z$. Then,  $\mathcal R$ is stable under multiplication by elements of $\mathcal F$ and $\mathcal C$, that is
$$
\mathcal F\cdot\mathcal R\subseteq\mathcal R,
\qquad
\mathcal C\cdot\mathcal R\subseteq\mathcal R.
$$
\end{corx}

Observe that Theorem~\ref{thm:main} applies to several natural classes of infinite subsets of $\mathbb Z$. In particular, every additive submonoid of $\mathbb Z$ occurs as a mapping degree set. This includes the sets $\mathbb Z_{\geq 0}$,  $\mathbb Z_{\leq 0}$, as well as arithmetic progressions of the form $
d\,\Z_{\ge0}=\{0,d,2d,3d,\ldots\},
$
where $d\neq 0$.  More generally, it  yields realizations of sets obtained from additive submonoids by finitely many sums and products, including finite unions of arithmetic progressions and the more elaborate families of sets described in Example~\ref{ex:monoid}.

Finally, we note that \cite[Problem 1.6]{NWW} is solved for arithmetic progressions contained entirely in the non-negative integers or entirely in the non-positive integers. The case of arithmetic progressions containing only finitely many positive integers or only finitely many negative integers remains open.

\subsection*{Funding} This work was partially supported by Agencia Estatal de Investigación (AEI) and the Ministerio de Ciencia, Innovación y Universidades (MICIU) {MCIN/AEI/10.13039/501100011033} grants PID2023-149804NB-I00 to C.C., B.V. and A.V., PREP2023-001879 to B.V-M., and PID2024-156578NB-I00 to V.M. 

\subsection*{Declaration} The authors have not used any AI tools or technologies to prepare this manuscript. 

\section{Additive submonoids of $\Z$} \label{sec:combinatory}

This section is devoted to the characterization of additive submonoids of $\Z$. We show that every such submonoid can be expressed as a finite sum of sets of the form
\[
d\,\Z_{\ge0}=\{0,d\}\Z_{\ge0}=\{0,d,2d,3d,\ldots\}.
\]
This follows from the fact that additive submonoids of $\Z$ are finitely generated (see \cite[Corollary~2.8]{Rosales}). For the reader’s convenience, we include a proof of this result.

\begin{proposition}
\label{prop:decomposition}
Let $A$ be an additive submonoid of $\Z$. Then there exist nonzero integers
$d_1,\ldots,d_m\in\Z$ such that
\[
A=\sum_{i=1}^m d_i\,\Z_{\ge0}.
\]
\end{proposition}

\begin{proof}
    Suppose that $A$ only contains 
    nonnegative integers and write $d =\gcd(A)$. Then $A\subset d\, \Z_{\ge 0}$.
    By Noetherianity, there exists a finite number of nonzero elements
$d_1,\ldots,d_k\in A$ such that
$
d=\gcd(d_1,\ldots,d_k).
$
    Reordering them, we can write Bézout's identity as
    $$
    a_1d_1+\ldots+a_rd_r-a_{r+1}d_{r+1}-\ldots-a_{k}d_{k}=d,
    $$
    for nonnegative integers $a_1,\ldots,a_k\geq 0$, 
    and some $0\leq r\leq k$. Let $A'\subset \Z$ be the submonoid generated by $\{d_1,\ldots,d_k\}\subset A$. So
    $A'\subset A$.
    Then $s:=a_{r+1}d_{r+1}+\ldots+a_{k}d_{k}\in A'$ and
    $s+d=a_1d_1+\ldots+a_{r}d_{r}\in A'$. 
If $s=0$ then $d\in A'$, so $d\,\Z_{\ge0}\subset A'\subset A\subset d\,\Z_{\ge0}$, hence $A'=  A$, and $A$ is finitely generated.
    
    If $s>0$, let $n\geq (s-1)s$. By Euclidean division we can write  
    $n=qs+r$ with $s>r\geq 0$. Then
    $qd\geq q\geq s-1\geq r$ and therefore
    $$
    nd=qsd-rs+rs+rd=(qd-r)s+r(s+d)\in A'.
    $$
    Therefore, every multiple $nd$ with $n\geq (s-1)s$ belongs to $A'$. Therefore 
    $A\cap d\,\Z_{\ge(s-1)s}=A'\cap d\,\Z_{\ge(s-1)s}= d\,\Z_{\ge(s-1)s}$. This implies that  
\[
\begin{aligned}
A
&=
\sum_{\substack{d'\in A\\ d'\geq (s-1)s\,d}}
d'\,\Z_{\ge0}
+\sum_{\substack{d'\in A\\ d'< (s-1)s\,d}}
d'\,\Z_{\ge0}\\
&=
A'
+\sum_{\substack{d'\in A\\ d'< (s-1)s\,d}}
d'\,\Z_{\ge0}\\
&=
\sum_{i=1}^k d_i\,\Z_{\ge0}
+\sum_{\substack{d'\in A\\ d'< (s-1)s\,d}}
d'\,\Z_{\ge0},
\end{aligned}
\]
hence $A$ is finitely generated. 

 Suppose now that $A$ contains only nonpositive integers. 
    Then $-A$ is a submonoid of $\Z$ that only contains nonnegative integers and, by the former case, 
    $-A=\sum\limits_{i=1}^md_i\,\Z_{\ge0}$, for some $d_1,\ldots,d_m\geq 0$, so $A=\sum\limits_{i=1}^m(-d_i)\,\Z_{\ge0}$.  Finally, if $A$ contains both 
    positive and negative integers, consider the submonoids of $A$ defined by 
    $A^+:=\{d\in A\,|\,d\geq 0\}$ and $A^-:=\{d\in A\,|\,d\leq 0\}$. 
    The result follows from the previous cases and the equality $A=A^++A^-$. 

\end{proof}

The following examples illustrate the richness of the class of subsets of $\mathbb Z$ generated by finite products of elements from the classes $\mathcal F$ and $\mathcal C$ defined in the Introduction.

\begin{example}\label{ex:monoid}
The following infinite subsets of $\Z$ are realizable as mapping degree sets.
\begin{enumerate}
    \item\label{ej:1} Additive submonoids $A$ of $\Z$, for instance:
    \begin{enumerate}
        \item $A= \Z_{\geq 0}$ or $A=\Z_{\leq 0}$.
        \item Arithmetic progressions starting at $0$,
        $
        d\,\Z_{\geq 0}=\{0,d,2d,3d,\ldots\},
        $
        where $d\neq 0$.
        \item Arithmetic progressions starting at $0$ with the first $(m-1)$ nonzero elements removed:
        $
        A=\{0,md,(m+1)d,(m+2)d,\ldots\},
        $
        where $d\neq 0$ and $m>1$.
    \end{enumerate}
    
    \item\label{ej:2} Finite unions of arithmetic progressions starting at $0$.
    
    \item\label{ej:3} The complement, together with $0$, of a symmetric interval:
    $
    C=\{n\in\Z: |n|>b\}\cup\{0\},
    $
    where $b\geq 1$.
    
    \item\label{ej:4} The composite numbers together with $0$:
    $
    C=\{n\in\Z_{\geq 0}: n\neq 1,\ n \text{ not prime}\}.
    $
    
    \item\label{ej:5} The complement in $\Z_{\geq 0}$ of $1$ and the prime numbers less than some $q\geq 1$. Namely
    $
    C=\{n\in\Z_{\geq 0}: n\neq 1,\ n\neq p \text{ prime with } p<q\}.
    $
\end{enumerate}
The realizability of the above sets follows from Theorem~\ref{thm:main} as follows. The sets in \eqref{ej:1} are additive submonoids of $\Z$. For \eqref{ej:2}, if
$
A_i=d_i\,\Z_{\geq 0},
$
$1\leq i\leq r$, then
\[
\bigcup_{i=1}^r A_i
=
\{0,d_1,\ldots,d_r\}\cdot \Z_{\geq 0},
\]
which is a product of a finite set and an additive submonoid.
For \eqref{ej:3}, letting
$
A=\{0,b+1,b+2,\ldots\},
$
we obtain
\[
C=\{n\in\Z: |n|>b\}\cup\{0\}
=
\{0,1,-1\}\cdot A,
\]
again a product of a finite set and an additive submonoid. For \eqref{ej:4},
\[
C=
\{0,2,3,4,\ldots\}\cdot \{0,2,3,4,\ldots\},
\]
which is a product of additive submonoids. Finally, if
$
A=\{0,q,q+1,q+2,\ldots\},
$
then the set $C$ in \eqref{ej:5} satisfies
\[
C=A+C',
\]
where
$
C'=\{n\in\Z_{\geq 0}: n\neq 1,\ n \text{ not prime}\}.
$

Thus each of the sets in \eqref{ej:2}--\eqref{ej:5} belongs to the class of subsets of $\mathbb Z$ generated by finite products of elements from the classes $\mathcal F$ and $\mathcal C$. 
The finiteness assumption in \eqref{ej:2} cannot be removed. Indeed, there exist countable unions of arithmetic progressions that are not recursively enumerable, and hence cannot occur as mapping degree sets.
\end{example}

\section{Realization of $\Z_{\ge 0}$ 
as mapping degree set}\label{sec:z>=0}

Our next objective is to describe a pair of closed connected manifolds
$M,N$ such that $D(M,N)=\Z_{\ge 0}$.
To realize this additive monoid, we consider manifolds studied by M\"ullner in \cite[Section 3.1]{MullnerAGT} and refer to \cite[Section 3.1]{MullnerThesis} for additional details.

For any given integer $n> 1$, let $T^{2n}=(S^1)^{2n}\subset {\mathbb C}^{2n}$ and define $f\colon T^{2n}\to T^{2n}$ as the orientation preserving diffeomorphism given by 
$$
f(z_1,z_2,\ldots,z_{2n-1},z_{2n})=(z_2, z_3,\ldots, z_{2n},z_1^{-1}z_2).
$$ 
This means that the action of $f$ on the group $H_1(T^{2n})=\Z^{2n}$ is given 
by the linear map 
$F\colon {\mathbb Z}^{2n}\to {\mathbb Z}^{2n}$ 
with matrix 
 \begin{equation}\label{eqn:F}
 F=
\left(\begin{array}{c|c}
  0   & \\
  \vdots & I_{2n-1}\\
  0 & \\ \hline
  -1   & \begin{array}{c|c} 1 & 0\, \cdots \,0\end{array}
\end{array}\right)
.
\end{equation}

Let $X_n$ be the mapping torus of $f$, that is, 
\begin{equation}\label{eq:mappingtorus}
X_n:=\bigl(T^{2n}\times [0,1]\bigr)/\sim\, ,
\end{equation}
where $(t,0)\sim (f(t),1)$. When the dimension is not relevant, we write $X$ instead of $X_n$.

Notice that the mapping torus $X$ is a $T^{2n}$-fiber bundle over $S^1$, that is, the natural projection $X\to \big([0,1]/0\sim 1\big)\cong S^1$ gives rise to a fiber sequence
\begin{equation}\label{eq:fibra_ de_X}
T^{2n}\longrightarrow X\longrightarrow S^1.
\end{equation}
Then the long exact sequence of homotopy groups associated with \eqref{eq:fibra_ de_X} shows that $X$ is aspherical, that is, $\pi_r(X)=0$ for $r>1$.

\begin{proposition}\label{prop:mullner}
The set of self-mapping degrees of the $(2n+1)$-manifold $X$ satisfies
 $$
 \{x^{2n}:x\in\mathbb{Z}\}\subset D(X,X)\subset \Z_{\ge0}.
 $$
\end{proposition}

\begin{proof}
Combining \cite[Proposition 27]{MullnerThesis} and \cite[Lemma 28]{MullnerThesis} we obtain that
\begin{align*}
 D(X,X)& =\{d\det(G): G\in \mathrm{Mat}_{2n\times 2n}(\mathbb{Z}), GF=F^d G\}\\
& =\{0\}\cup\{d\det(G): G\in \mathrm{Mat}_{2n\times 2n}(\mathbb{Z})\cap \mathrm{GL}_{2n}(\mathbb{Q}), \, GFG^{-1}=F^d\}.
\end{align*}
Therefore, describing $D(X,X)$ requires identifying the integers $d$ such that $F$ and $F^d$ 
are similar matrices, and then 
describing the integral changes of basis 
$G$ such that $GFG^{-1}=F^d$. 

First, to decide whether $F$ and $F^d$ are similar matrices, consider the characteristic polynomial of $F$ which is $p(x)=\chi_F(x)=x^{2n}-x+1$. 
A direct computation shows that the global minimum of $p$ is
$p\!\left(\left(\frac1{2n}\right)^{\frac1{2n-1}}\right)>0$,
and hence $p$ has no real roots. Furthermore, all roots of $p$ are simple. Indeed, if $\lambda$ were a multiple root of $p$, then $p(\lambda)=p'(\lambda)=0$. Since $p'(x)=2n\,x^{2n-1}-1$, it follows that $\lambda^{2n-1}=1/(2n)$ and thus $\lambda^{2n}=\lambda/(2n)$. Substituting into $p(\lambda)=0$ gives $\lambda=2n/(2n-1) \in \mathbb R$, contradicting the fact that $p$ has no real roots. 
Therefore  
$p$
has exactly $2n$ pairwise different complex
roots $\{\lambda_1,\overline{\lambda_1},\ldots,\lambda_n, \overline{\lambda_n}\}$. 

Now, the polynomial
whose roots are  
$\{\lambda_1^{-1},\overline{\lambda_1}{}^{-1},\ldots,\lambda_n^{-1}, \overline{\lambda_n}{}^{-1}\}$ is
$$q(x)=x^{2n}p(x^{-1})
=x^{2n}-x^{2n-1}+1\neq p(x),$$ since $n>1$. 
Hence there must exist at least one root, 
say $\lambda_n$, such that 
$p(\lambda_n^{-1})\neq 0$. 
Thus $\lambda_n^{-1}\neq \overline{\lambda_n}$ and hence $|\lambda_n|\neq 1$. Since 
$\det(F)=p(0)=\prod |\lambda_i|^2=1$, we may assume 
that $|\lambda_n|>1$.

Now, assume $F$ and $F^d$ are similar matrices. Then $F$ and $F^{d^r}$ are also similar for any $r\geq 0$ and 
$\chi_F=\chi_{F^{d^r}}$. Therefore, the eigenvalues of $F^{d^r}$, namely 
$\{\lambda_1^{d^r},\overline{\lambda_1}{}^{d^r},\ldots,\lambda_n^{d^r}, \overline{\lambda_n}{}^{d^r}\}$, must be roots of $\chi_F$ for every $r\geq 1$, that is
$$
\{\lambda_1^{d^r},\overline{\lambda_1}{}^{d^r},\ldots,\lambda_n^{d^r}, \overline{\lambda_n}{}^{d^r}\}=\{\lambda_1,\overline{\lambda_1},\ldots,\lambda_n, \overline{\lambda_n}\},
$$
for any $r\geq 1$.
But, if $|d|>1$, $\lim\limits_{r\to\infty}|\lambda_n^{d^{r}}|=\infty$ while $\chi_F$ has finitely many roots. So if $F$ and $F^d$ are similar matrices, then $|d|=1$.
Moreover, $F$ and $F^{-1}$ are not similar
since $\chi_{F^{-1}}(x)=q(x)\neq p(x)=\chi_F(x)$, 
hence if $F$ and $F^d$ are similar matrices, then $d=1$.

Second, if $GFG^{-1}=F$, then $G$ centralizes $F$ and, since all eigenvalues of $F$ are pairwise different complex numbers, $G=P(F)$ where $P$ is a polynomial with coefficients in $\mathbb{Q}$ (see \cite[Lemma 30]{MullnerThesis}). Therefore 
$$
\det(G)=\prod_{j=1}^n P(\lambda_j)
P(\,\overline{\lambda_j}\,)=
\prod_{j=1}^n|P(\lambda_j)|^2>0,
$$
and then $D(X,X) \subset\mathbb{Z}^{\geq 0}$. 

Finally, every integral diagonal matrix $G_x=x I_{2n} 
$, with $x\in\mathbb{Z}$, centralizes $F$ so 
$$
\{\det(G_x)=x^{2n} : x\in\mathbb{Z}\}\subset D(X,X),
$$
which completes the proof.
\end{proof}

\begin{remark}
The set \(D(X,X)\) is multiplicative and can be described as the set of values
of a polynomial evaluated at integer points.
Indeed, by \cite[Lemma 30]{MullnerThesis},
\[
C_{\mathbb Z}(F)
=
\mathrm{Mat}_{2n\times 2n}(\mathbb Z)
\cap
\{P(F):P\in\mathbb Q[x]\}.
\]
Since the characteristic polynomial of \(F\) has degree \(2n\), every polynomial
in \(F\) can be written, modulo \(\chi_F\), in the form
\[
a_0I+a_1F+\cdots+a_{2n-1}F^{2n-1},
\qquad a_0,\ldots,a_{2n-1}\in\mathbb Q.
\]
Thus
\[
C_{\mathbb Z}(F)
=
\left\{
a_0I+a_1F+\cdots+a_{2n-1}F^{2n-1} \in
\mathrm{Mat}_{2n\times 2n}(\mathbb Z)
:
a_i\in\mathbb Q\right\}.
\]
In this particular case, the explicit form of these matrices shows that the
integrality condition is equivalent to
\[
a_0,\ldots,a_{2n-1}\in\mathbb Z.
\]
Consequently,
\[
D(X,X)
=
\left\{
\det(a_0I+a_1F+\cdots+a_{2n-1}F^{2n-1})
:
a_0,\ldots,a_{2n-1}\in\mathbb Z
\right\}.
\]
Since this determinant is a polynomial in the variables
\(a_0,\ldots,a_{2n-1}\), the set \(D(X,X)\) is the set of values of this
polynomial on \(\mathbb Z^{2n}\).
\end{remark}

\if{}
\begin{remark}
    In fact, $D(X,X)$ is a multiplicative set consisting of the values of a polynomial at the integers. Indeed, the center $C_\Z(F)$ is equal to $\mathrm{Mat}_{2n\times 2n}(\mathbb{Z})\cap \{P(F):P\in\Q[x]\}$ (see \cite[Lemma 30]{MullnerThesis}). It is easy to check that $\{P(F):P\in\Q[x]\}$ is the set of matrices of the form
     \begin{multline*}\label{eqn:C(F)}
 G(a_0,\ldots,a_{2n-1})=a_0 I_{2n}+a_1
\left(\begin{array}{c|c}
  0   & \\
  \vdots & I_{2n-1}\\
  0 & \\ \hline
  -1   & \begin{array}{c|c} 1 & 0\, \cdots \,0\end{array}
\end{array}\right)
+\\+a_2\left(\begin{array}{c|c}
  \begin{array}{cc} 0 & 0\\\vdots&\vdots\\0&0\end{array}& I_{2n-2}\\ \hline
  \begin{array}{cc} -1 & 1\\0&-1\end{array} & \begin{array}{c|c} 0 & 0\, \cdots \,0\\1 & 0\, \cdots \,0\end{array}
\end{array}\right)+\cdots+a_{2n-1}\left(\begin{array}{c|c}
  \begin{array}{ccccc} 0 & \cdots& \cdots& \cdots&0\end{array}& 1\\ \hline
  \begin{array}{cccccc} -1 & 1&0 & \cdots&0\\0&-1&1&\cdots&0\\0&0&-1&\ddots&\vdots\\
  \vdots&\vdots&\ddots&\ddots &1\\0&0&\cdots&0 &-1\end{array} & \begin{array}{c} 0 \\\vdots\\\vdots\\0\\ 1\end{array}
\end{array}\right)=\\
=\left(\begin{array}{cccccc}
a_0&a_1&a_2&\cdots &a_{2n-1}\\
-a_{2n-1}&a_0+a_{2n-1}&a_1&\ddots&\vdots\\
-a_{2n-2}&a_{2n-2}-a_{2n-1}&a_0+a_{2n-1}&\ddots&a_2\\
\vdots&\vdots&\ddots&\ddots&a_1\\
-a_1&a_1-a_2&\cdots&a_{2n-2}-a_{2n-1}&a_0+a_{2n-1}
\end{array}\right),
\end{multline*}
with $a_0,\ldots,a_{2n-1}\in\Q$. Clearly $G(a_0,\ldots,a_{2n-1})\in \mathrm{Mat}_{2n\times 2n}(\mathbb{Z})$ if and only if $a_0,\ldots,a_{2n-1}\in\Z$, hence $D(X,X)$ is the multiplicative set $\{\det(G(a_0,\ldots,a_{2n-1})):a_0,\ldots,a_{2n-1}\in\Z\}$.
\end{remark}
\fi

We now recall a well-known additivity property of mapping degree sets under connected sums, which will be used repeatedly throughout the paper.
\begin{lemma}[{\cite[Lemma 7.8]{CMuV1}, \cite[Lemma 3.5]{NWW}}]\label{lem:sum}
    Let $M_1,\ldots, M_k$ and $N$ be closed oriented connected manifolds of dimension $m$. Then
    $$
    \sum_{i=1}^k D( M_i,N)\subset D(\#_{i=1}^k M_i,N).
    $$
    Moreover, if $\pi_{m-1}(N)=0$ then
    $$
    \sum_{i=1}^k D( M_i,N)= D(\#_{i=1}^k M_i,N).
    $$
\end{lemma}


By now applying this lemma to the mapping torus $X$ constructed in \eqref{eq:mappingtorus}, we obtain the following realization result.
\begin{theorem}\label{thm:realizando_z+}
For every $n>1$, there exists an integer $s =s_n$ 
which depends on $n$ such that the $(2n+1)$-manifolds $X^{\#s}$ and $X$ satisfy $D(X^{\#s},X)=\Z_{\ge0}$.
\end{theorem}

\begin{proof}

\if{}
Let $s=s(2n)$ denote the minimum number such that every 
positive integer can be expressed as a sum of at most 
$s$ numbers which are $2n$-th powers of integers. The 
number $s(2n)$ exists by the Hilbert-Waring Theorem \cite{Hilbert-Waring}.

Since $\pi_{2n-1}(X)=0$, 
we can apply Lemma \ref{lem:sum} to get that
 $$
 D(X^{\#s},X)=\sum_{i=1}^{s} D(X,X).
 $$
On one hand $D(X,X)\subset \Z_{\ge0}$, hence $D(X^{\#s},X)\subset \Z_{\ge0}$. On the other hand, 
$\{x^{2n}:x\in\mathbb{Z}\}\subset D(X,X)$. As every positive integer $d$ can be written as
$d=\sum\limits_{i=1}^s x_i^{2n} \in 
\sum\limits_{i=1}^s D(X,X)$, 
we have that $d\in D(X^{\#s},X)$. Hence 
$D(X^{\#s},X)=\Z_{\ge0}$.
\fi

Let \(s=s(2n)\) be the smallest integer with the property that every positive
integer can be expressed as a sum of at most \(s\) integer \(2n\)-th powers.
The existence of such an \(s\) is guaranteed by the Hilbert--Waring Theorem
\cite{Hilbert-Waring}.

Since \(\pi_{2n}(X)=0\), Lemma \ref{lem:sum} yields
\[
D(X^{\# s},X)=\sum_{i=1}^{s}D(X,X).
\]
By Proposition \ref{prop:mullner}, we have \(D(X,X)\subset \mathbb Z_{\geq 0}\),
and therefore
\[
D(X^{\# s},X)\subset \mathbb Z_{\geq 0}.
\]

Moreover,
\[
\{x^{2n}:x\in\mathbb Z\}\subset D(X,X).
\]
Let \(d\in\mathbb Z_{>0}\). By the definition of \(s\), there exist integers
\(x_1,\ldots,x_s\) such that
\[
d=x_1^{2n}+\cdots+x_s^{2n}.
\]
Since each \(x_i^{2n}\) belongs to \(D(X,X)\), it follows that
\[
d\in \sum_{i=1}^{s}D(X,X)
   =D(X^{\# s},X).
\]
Thus \(\mathbb Z_{>0}\subset D(X^{\# s},X)\),  and consequently
\[
D(X^{\# s},X)=\mathbb Z_{\geq 0}.
\]

\end{proof}
\section{Realization of submonoids as mapping degree sets} \label{sec:submonoids}

In this section we construct, for every $d\in\Z$,  manifolds whose mapping degree set is $d\,\Z_{\geq 0}$. As an application, we show that every additive submonoid of $\Z$ is a mapping degree set. The construction is based on a family of aspherical manifolds that we now introduce.

\subsection{The manifolds $E_i$}
Let $\Sigma_g$ denote the closed oriented hyperbolic surface of genus $g>1$. Since $\Sigma_g$ is hyperbolic, it is aspherical.
For $i\in \Z\smallsetminus{0}$, let $E_{i,g}$ denote the total space of the circle bundle
\begin{equation}\label{eqn:Ei}
S^1\longrightarrow E_{i,g}\longrightarrow \Sigma_g
\end{equation}
with Euler number $i$. 
The long exact sequence of homotopy groups associated with the bundle shows that $E_{i,g}$ is also aspherical. 
The fundamental group of $E_{i,g}$ is an extension 
 \begin{equation}\label{eqn:kerpj}
  \pi_1(S^1)=\Z \to \pi_1(E_{i,g}) \to H=\pi_1(\Sigma_g).
      \end{equation}
We write $\ell$ for the generator of $\pi_1(S^1)$
and $a_1,b_1,\ldots, a_g,b_g$ for the generators of
$\pi_1(\Sigma_g)$. Then $\pi_1(E_{i,g})$ is
generated by $\ell,a_1,b_1,\ldots, a_g,b_g$ and satisfies the relation
 $$
 \prod_{k=1}^g [a_k,b_k] =\ell^i \, .
 $$
The description of $\pi_1(E_{i,g})$ immediately yields the first homology group:
$$
H_1(E_{i,g},\Z)
=
\pi_1(E_{i,g})_{ab}
=
\Z^{2g}\times(\Z/i\,\Z),
$$
where the free part is generated by
$a_1,b_1,\ldots,a_g,b_g$ and the torsion factor by $\ell$.
Hence
$$
H_1(E_{i,g},\Q)\cong\Q^{2g}.
$$

\medskip 

We will use the following computation of mapping degree sets due to Neofytidis, Wang, and Wang. For simplicity, we write $E_i$ instead of $E_{i,g}$ whenever the genus is understood. 

\begin{lemma}[{\cite[Lemma 3.4]{NWW}}]\label{lem:Ei}
If $i,j\in\Z\smallsetminus\{0\}$, then
$$
D(E_i,E_j)=
\begin{cases}
\{0,\frac{j}{i}\} & \text{if } i\mid j,\\
\{0\} & \text{if } i\nmid j.
\end{cases}
$$
\end{lemma}
In particular, for every $d\in\Z$, choosing $i\in\Z\smallsetminus\{0\}$, we obtain
$$
D(E_i,E_{id})=\{0,d\}.
$$

Thus, the manifolds $E_i$ realize the finite sets
$
\{0,d\}.
$
On the other hand, Theorem~\ref{thm:realizando_z+} provides manifolds realizing
$
\Z_{\geq0}.
$
Our goal is to combine these constructions in order to realize the additive submonoids
$
d\,\Z_{\geq0}.
$
The main tool will be a product formula for mapping degree sets.

\subsection{A product theorem}

To establish the required product formula, we will use the notion of domination by products: a manifold $N$ is not dominated by any nontrivial product if every map
$
f\colon M_1\times M_2\to N
$
from a product of manifolds satisfying
$
\dim(M_1)+\dim(M_2)=\dim(N)
$
has degree zero.

We now prove a generalization of \cite[Theorem 4.6]{NWW}, which will allow us to combine mapping degree sets multiplicatively.

\begin{theorem}\label{thm:product}
    Let $M_1$ and $M_2$ be closed oriented manifolds of dimension $m$.
Let $N_1,\ldots,N_k$ be closed oriented manifolds of dimensions
$n_1,\ldots,n_k$, and let $N$ be a closed oriented manifold of dimension
$n=\sum_{i=1}^k n_i$. Suppose that, for each $i\in\{1,\ldots,k\}$,
\begin{enumerate}
    \item\label{cond1} $N_i$ is not dominated by any nontrivial product;
    \item\label{cond2} for every map $f\colon M_1\rightarrow N_i$, there exists
    $p\geq 1$ such that the induced homomorphism
    $H_p(f)\colon H_p(M_1,\Q)\rightarrow H_p(N_i,\Q)$ is not surjective.
\end{enumerate}
Then
$$
D(M_1\times N,M_2\times \prod_{i=1}^k N_i)
=
D(M_1,M_2)\cdot D(N,\prod_{i=1}^k N_i).
$$
\end{theorem}

\begin{proof}
    The inclusion $$D(M_1,M_2)\cdot D(N,\prod_{i=1}^kN_i)\subset D(M_1\times N,M_2\times \prod_{i=1}^kN_i)$$
    is immediate.  
    
    For the reverse inclusion, let $f\colon M_1\times N\rightarrow M_2\times \prod_{i=1}^kN_i$ be a map of nonzero degree. 
 Fix $i_0\in\{1,\ldots,k\}$ and let
$
[N_{i_0}]^*
$
denote the cohomological fundamental class of \(N_{i_0}\).
We first show that
$$
H^{n_{i_0}}(f)(1\otimes [N_{i_0}]^*)
\in
H^0(M_1,\Q)\otimes H^{n_{i_0}}(N,\Q).
$$
Write
$
H^{n_{i_0}}(f)(1\otimes [N_{i_0}]^*)
=
\sum_j d_j(\alpha_j\otimes\beta_j),
$
where $d_j\in\Q\smallsetminus\{0\}$, and
$\alpha_j\in H^{q_j}(M_1,\Q)$ and
$\beta_j\in H^{n_{i_0}-q_j}(N,\Q)$ are homogeneous basis elements such that the pairs
$(\alpha_j,\beta_j)$ are pairwise distinct. Suppose first that there exists $j_0$ such that
$q_{j_0}\notin\{0,n_{i_0}\}$.
By \cite[Théorèmes III.4]{Thom}, there exist manifolds
$A$ and $B$ of dimensions
$q_{j_0}$ and $n_{i_0}-q_{j_0}$,
maps
$g_1\colon A\to M_1$
and
$g_2\colon B\to N$,
and nonzero rational numbers
$a$ and $b$
such that the nontrivial images of the induced homomorphisms $H^{q_{j_0}}(g_1)$ and $H^{n_{i_0}-q_{j_0}}(g_2)$ are given by $H^{q_{j_0}}(g_1)(\alpha_{j_0})=a[A]^*$ and $H^{n_{i_0}-q_{j_0}}(g_2)(\beta_{j_0})=b[B]^*$.  Let $$\pi_{N_{i_0}}\colon M_2\times \prod_{i=1}^k N_i\to N_{i_0}$$
be the projection onto the $N_{i_0}$-factor, and set
$$
h:=\pi_{N_{i_0}}\circ f\circ(g_1\times g_2)
\colon A\times B\to N_{i_0}.
$$
Then
$$
H^{n_{i_0}}(h)([N_{i_0}]^*)
=
a\,b\,d_{j_0}([A]^*\otimes [B]^*)
=
a\,b\,d_{j_0}[A\times B]^*.
$$
Thus 
$\deg(h)\neq 0$, contradicting Hypothesis \eqref{cond1}.

It remains to exclude the possibility that
$q_{j_0}=n_{i_0}.$
By \cite[Théorèmes III.4]{Thom}, there exists a manifold $C$ of dimension
$n_{i_0}$ and a map $g\colon C\to M_1$ detecting $\alpha_{j_0}$.
Let $i_{M_1}\colon M_1\rightarrow M_1\times N$ be the inclusion and
$$
h':=\pi_{N_{i_0}}\circ f\circ i_{M_1}  \colon M_1 \to  N_{i_0}.
$$
A computation analogous to the previous one shows that
 $\deg(h' \circ g)\neq 0$.

Since a map of nonzero degree induces surjective homomorphisms in rational homology,
$H_p(h'\circ g)$ is surjective for every $p\geq 1$. Hence
$H_p(h')$ is also surjective for every $p\geq 1$, contradicting
Hypothesis~\eqref{cond2}.

Therefore, for every $i\in\{1,\ldots,k\}$, we have
$$
H^{n_i}(f)(1\otimes [N_i]^*)\in
H^0(M_1,\Q)\otimes H^{n_i}(N,\Q).
$$
Hence there exists 
$d\in\Z$ 
such that
$$
H^n(f)\left(1\otimes \left[\prod_{i=1}^k N_i\right]^*\right)
=
d(1\otimes [N]^*)\in H^0(M_1,\Q)\otimes H^n(N,\Q).
$$

     Let $H^m(f)([M_2]^*\otimes 1)=c([M_1]^*\otimes 1)+\delta$, where $\delta\in \bigoplus_{j=0}^{m-1}H^j(M_1,\Q)\otimes H^{m-j}(N,\Q)$ and $c\in\Z$. Then
     $$
     H^{m+n}(f)\left(\left[M_2\times\prod_{i=1}^kN_i\right]^*\right)=(c([M_1]^*\otimes 1)+\delta)\smile d(1\otimes [N]^*)
     =c\,d[M_1\times N]^*,
     $$
     so $\deg(f)=c\,d$.

     Furthermore, let $i_N\colon N\rightarrow M_1\times N$ be the inclusion and let
$\pi_{\prod N_i}\colon M_2\times \prod_{i=1}^kN_i\rightarrow \prod_{i=1}^kN_i$
be the projection. Consider the map
$$
\pi_{\prod N_i}\circ f\circ i_N
\colon N\rightarrow \prod_{i=1}^kN_i.
$$
By construction,
$$
H^n(\pi_{\prod N_i}\circ f\circ i_N)\left(\left[\prod_{i=1}^kN_i\right]^*\right)
=
d[N]^*.
$$
Therefore,
$
\deg(\pi_{\prod N_i}\circ f\circ i_N)=d,
$
and hence
$$
d\in D\left(N,\prod_{i=1}^kN_i\right).
$$

     On the other hand, let
$\pi_{M_2}\colon M_2\times \prod_{i=1}^kN_i\rightarrow M_2$
be the projection. Then
$$
H^m(\pi_{M_2}\circ f\circ i_{M_1})([M_2]^*)
=
c[M_1]^*.
$$
Hence
$
\deg(\pi_{M_2}\circ f\circ i_{M_1})=c,
$
and therefore
$$
c\in D(M_1,M_2).
$$

Altogether,
$$
\deg(f)=cd\in D(M_1,M_2)\cdot D\!\left(N,\prod_{i=1}^kN_i\right).
$$
This proves the reverse inclusion.
\end{proof}
\if[]

The following corollary will allow us to apply Theorem \ref{thm:product} in a systematic way.

\begin{corollary}\label{cor:prod}
    Let $A\subset \Z$. Suppose that there exists a family of tuples of manifolds $\{(N_i,N_{i,1},\ldots,N_{i,k})\}_{i\in I}$ such that $\dim(N_i)=\sum_{j=1}^k\dim(N_{i,j})$ and $D(N_i,\prod_{j=1}^kN_{i,j})=A$ for every $i\in I$. Assume moreover that: 
    \begin{enumerate}
        \item $N_{i,j}$ is not dominated by any nontrivial product for all $i\in I$ and all $j\in\{1,\ldots,k\}$; 
        \item for every $n_1,\ldots,n_k>0$, there exist infinitely many $i\in I$ such that for each $j\in\{1,\ldots,k\}$, either
        \begin{enumerate}[(i)]
            \item $\dim(N_{i,j})\geq n_j$  or
        \item there exists $p_j\geq 1$such that $\dim(H_{p_j}(N_{i,j},\Q))\geq n_j$.
        \end{enumerate}
    \end{enumerate}
    Then, for any closed oriented manifolds $M_1$ and $M_2$ of the same dimension, there exist infinitely many $i\in I$ such that $$D(M_1\times N_i,M_2\times \prod_{j=1}^kN_{i,j})=A\cdot D(M_1,M_2).$$
\end{corollary}

\begin{proof}
    The corollary follows from Theorem \ref{thm:product} by choosing all $i\in I$ such that, for every $j\in\{1,\ldots,k\}$, $\dim(N_{i,j})>\dim(M_1)$ or $\dim(H_{p_j}(N_{i,j},\Q))>\dim(H_{p_j}(M_1,\Q))$.
\end{proof}
\fi

\subsection{The manifolds $X_n$} 

We now return to the manifolds $X_n$ defined in \eqref{eq:mappingtorus}. Together with the manifolds $E_i$, they will allow us to realize the additive submonoids $d\,\Z_{\geq0}$ via Theorem~\ref{thm:product}. To apply that theorem, we need to show that, under certain restrictions, $X_n$ is not dominated by a nontrivial product.

We shall establish this through the notion of presentability by products introduced in \cite{KL}:
    an infinite group $G$ is not presentable by a product if, for each homomorphism $\varphi\colon  H_1\times H_2\rightarrow G$ onto a subgroup of finite index, $\varphi(H_i)\subset G$ is finite for some $i=1,2$.

We begin by studying its fundamental group. Our goal is to show that if $n\not\equiv 1 \pmod 3$, $\pi_1(X_n)$ is not presentable by a product and then deduce that $X_n$ is not dominated by a nontrivial product. The fundamental group of $X_n$ fits into the short exact sequence
\begin{equation}\label{eqn:K}
\pi_1(T^{2n})=\Z^{2n}\longrightarrow K=\pi_1(X_n)\longrightarrow \pi_1(S^1)=\Z.
\end{equation}
Let $\{x_1,\ldots,x_{2n}\}$ be the standard basis of $\Z^{2n}$, and let $x_{2n+1}$ be a generator of $\Z$. Then
$$
\pi_1(X_n)=\Z^{2n}\rtimes \Z,
$$
where the action of $\Z$ on $\Z^{2n}$ is determined by
$$
x_{2n+1}\Bigl(\prod_{i=1}^{2n}x_i^{\alpha_i}\Bigr)x_{2n+1}^{-1}
=
\prod_{i=1}^{2n}x_i^{\beta_i}
\quad\Longleftrightarrow\quad
F
\begin{pmatrix}
\alpha_1\\
\vdots\\
\alpha_{2n}
\end{pmatrix}
=
\begin{pmatrix}
\beta_1\\
\vdots\\
\beta_{2n}
\end{pmatrix},
$$
and $F$ is the matrix defined in \eqref{eqn:F}.


\begin{lemma}\label{lem:notpresprod}
    If $n\not\equiv 1 \pmod 3$, then the fundamental group of $X_n$ is not presentable by a product.
\end{lemma}

\begin{proof}
By \cite[Example 3.8]{KL2}, the fundamental group of a $T^m$-bundle over $S^1$ is not presentable by a product whenever $1$ is not an eigenvalue of any power of the matrix $F$ that defines the action of $\pi_1(S^1)=\Z$ on $\pi_1(T^m)=\Z^{m}$. Therefore, it suffices to show that no eigenvalue of $F$ is a root of unity.

Assume, for contradiction, that $p(x)=\chi_F(x)=x^{2n}-x+1$ has a root $\lambda$ which is a root of unity. Then $\lambda=\overline{\lambda}^{-1}$, and therefore $\lambda$ is also a root of
$
q(x)=x^{2n}p(x^{-1})=x^{2n}-x^{2n-1}+1.
$
Hence $p(x)$ and $q(x)$ are not coprime. However,
\begin{align*}
\gcd(p(x),q(x))
&=
\gcd(x^{2n}-x+1,x^{2n-1}-x)\\
&=
\gcd(x^{2n}-x+1,x^{2n-2}-1)\\
&=
\gcd(x^{2n}+x^{2n-2}-x,x^{2n-2}-1)\\
&=
\gcd(x^{2n-1}+x^{2n-3}-1,x^{2n-2}-1)\\
&=
\gcd(x^{2n-1}+x^{2n-3}-x^{2n-2},x^{2n-2}-1)\\
&=
\gcd(x^2-x+1,x^{2n-2}-1).
\end{align*}
Notice that $x^3+1=(x^2-x+1)(x+1)$ and then the roots of $x^2-x+1$ are the $6$-th primitive roots of unity. Therefore, $\gcd(x^2-x+1,x^{2n-2}-1)\neq 1$ if and only if 
$6\mid(2n-2)$, that is, $3\mid(n-1)$, which contradicts the hypothesis. Thus, if  $n\not\equiv 1 \pmod 3$, then no eigenvalue of $F$ is a root of unity, and the result follows.
\end{proof}
We now pass from the fundamental group to the manifold itself, using the relationship between presentability by products and domination by products established in \cite{KL}.
\begin{lemma}\label{lem:Xnotdomprod}
    The mapping torus $X_n$ is not dominated by a nontrivial product for each $n\not\equiv 1 \pmod 3$.
\end{lemma}

\begin{proof}
Since $X_n$ is aspherical, it is rationally essential, that is,
$$
H_{2n+1}(c)([X_n])\neq 0
\in
H_{2n+1}(B\pi_1(X_n),\Q),
$$
where
$
c\colon X_n\rightarrow B\pi_1(X_n)
$
classifies the universal cover of $X_n$.
Together with Lemma \ref{lem:notpresprod}, the result follows from \cite[Theorem 1.4]{KL}.
\end{proof}

\subsection{Stability under multiplication by $\Z_{\geq0}$}

An appropriate choice of the manifolds $X_n$ constructed in Theorem~\ref{thm:realizando_z+} satisfies the hypotheses of Theorem~\ref{thm:product}. As a consequence, multiplication by the mapping degree set
$
D(X_n^{\#s_n},X_n)=\Z_{\geq0}
$
preserves mapping degree sets.

\begin{proposition}\label{cor:dZ>=0}
   Let $M$ and $N$ be closed oriented connected manifolds of the same dimension, and let $n\not\equiv 1 \pmod 3$ be such that 
$
2n+1>\dim(M).
$
Then
$$
D(M\times X_n^{\# s_n},N\times X_n)
=
D(M,N)\cdot\Z_{\geq 0}.
$$

\end{proposition}

\begin{proof}
Let
$
f\colon M\to X_n
$
be an arbitrary map. Since $\dim(X_n)=2n+1>\dim(M)$, Hypothesis \eqref{cond2} of Theorem \ref{thm:product} is automatically satisfied. Moreover, by Lemma \ref{lem:Xnotdomprod}, $X_n$ is not dominated by any nontrivial product. Hence Theorem \ref{thm:product} gives
$$
D(M\times X_n^{\# s_n},N\times X_n)
=
D(M,N)\cdot D(X_n^{\# s_n},X_n).
$$
By Theorem \ref{thm:realizando_z+},
$$
D(X_n^{\# s_n},X_n)=\Z_{\geq 0}.
$$
Therefore,
$$
D(M\times X_n^{\# s_n},N\times X_n)
=
D(M,N)\cdot\Z_{\geq 0}.
$$
\end{proof}
Combining Lemma~\ref{lem:Ei} with Proposition~\ref{cor:dZ>=0}, we obtain the following realization of the additive submonoids $d\,\Z_{\geq0}$.
\begin{corollary}\label{cor:dZ}
 For every $d\in\Z$, the additive submonoid $d\Z_{\geq 0}$ is realizable as a mapping degree set. More precisely, for every $i\in\Z\smallsetminus\{0\}$ and every $n\not\equiv 1 \pmod 3$,
$$
D(E_i\times X_n^{\# s_n},E_{id}\times X_n)
=
\{0,d\}\cdot\Z_{\geq 0} = d\,\Z_{\geq 0}.
$$
\end{corollary}

The submonoids $d\,\Z_{\geq0}$ constitute the basic building blocks for arbitrary additive submonoids of $\Z$. Using Proposition~\ref{prop:decomposition} together with Lemma~\ref{lem:sum}, we obtain the following result. 
\begin{corollary} \label{cor:additive}
Let $A\subset \Z$ be an additive submonoid.
Then, there exist closed
oriented connected manifolds $M$ and $N$, with $N$ aspherical,  such that
$D(M,N)=A$.
\end{corollary}

\begin{proof} 
    By Proposition \ref{prop:decomposition}, there exist (nonzero) $d_1,\ldots,d_k\in \Z$ such that 
    $A=\sum\limits_{l=1}^kd_l\Z_{\ge0}$. Let $j=\prod\limits_{m=1}^k d_m$  
    and $i_l =\prod\limits_{m\neq l} d_m$, 
    and define $N=E_j\times X$ and $M_l =E_{i_l}\times X^{\#s}$, for $l=1,\ldots,k$. Since $N$ is aspherical, by Proposition \ref{cor:dZ>=0} and Lemma \ref{lem:sum} we have 
$$
D(\#_{l=1}^k M_l,N)=\sum_{l=1}^k D( M_l,N)=\sum_{l=1}^kd_l\,\Z_{\ge0}=A\, .
$$
\end{proof}
In the notation of the Introduction, Corollary~\ref{cor:additive} shows that every element of $\mathcal C_0$ is realizable. This constitutes the first step towards Theorem~\ref{thm:main}, which establishes the analogous statement for every element of $\mathcal C$.

\section{Multiplicative Stability and Proof of Theorem~\ref{thm:main}}\label{sec:plus}

In this section we prove Theorem~\ref{thm:main}. The main point is not merely that every set in $\mathcal F$, the class of finite sets of $\Z$ containing $0$,  and every set in $\mathcal C$, the class of sets obtained from additive submonoids of $\Z$ by finitely many sums and products,  occurs as a mapping degree set. Rather, we establish a stronger multiplicative stability property: multiplying an arbitrary mapping degree set by any set belonging to either class again yields a mapping degree set.

\medskip 

A key ingredient in the proof is the following:
\begin{lemma}[{\cite[Theorem 1]{Wang}, \cite[Lemma 1]{KN}}]\label{lem:Einotdomprod}
The manifold $E_i$ is not dominated by any nontrivial product for all $i\neq 0$.
\end{lemma}

The proof of Theorem~\ref{thm:main} is divided into two parts. We first treat finite sets, corresponding to the class $\mathcal F$, and then consider sets belonging to the class $\mathcal C$.

\if{}

Before that, we recall some properties of the manifolds $E_{i,g}$, where $i\neq 0$. The fundamental group of $E_{i,g}$ is an extension 
 \begin{equation}\label{eqn:kerpj}
  \pi_1(S^1)=\Z \to \pi_1(E_{i,g}) \to H=\pi_1(\Sigma_g).
      \end{equation}
We write $\ell$ for the generator of $\pi_1(S^1)$
and $a_1,b_1,\ldots, a_g,b_g$ for the generators of
$\pi_1(\Sigma_g)$. Then $\pi_1(E_{i,g})$ is
generated by $\ell,a_1,b_1,\ldots, a_g,b_g$ and satisfies the relation
 $$
 \prod_{k=1}^g [a_k,b_k] =\ell^i \, .
 $$
The first integral homology group of $E_{i,g}$ is
$$
H_1(E_{i,g},\Z)
=
\pi_1(E_{i,g})_{ab}
=
\Z^{2g}\times(\Z/i\,\Z),
$$
where the free part is generated by
$a_1,b_1,\ldots,a_g,b_g$ and the torsion factor by $\ell$.
Hence
$$
H_1(E_{i,g},\Q)\cong\Q^{2g}.
$$

\fi

\begin{theorem}\label{thm:proofmainf}
Let \(A \) be a set in  $ \mathcal F$, and let $M_1$ and $M_2$ be a pair of closed oriented connected manifolds of the same dimension. Then there exist closed oriented connected manifolds $N_1,N_2$ such that:
\begin{enumerate}
\item $D(N_1,N_2)=A.
$
\item 
$D(M_1\times N_1,M_2\times N_2)
=
D(M_1,M_2)\cdot A.$
\end{enumerate}

\end{theorem}

\begin{proof}
By \cite[Theorem A]{CMuV2}, for every $g>1$ there exist manifolds $N_1$ and
$
N_2=E_{i_1,g}\#\cdots\#E_{i_k,g},
$
with $i_1,\ldots,i_k\neq 0$, such that
$
D(N_1,N_2)=A.
$

We claim that $N_2$ is not dominated by any nontrivial product. Indeed, if there were a map of nonzero degree from a nontrivial product to $N_2$, then composing with the pinch map
$
N_2\longrightarrow E_{i_1,g},
$
which has degree $1$ by \cite[Lemma 4.2]{NWW}, would yield a map of nonzero degree from a nontrivial product to $E_{i_1,g}$. This contradicts Lemma \ref{lem:Einotdomprod}.

Moreover,
$
\dim H_1(N_2,\Q)\geq 2g.
$
As $g$ can be taken arbitrarily large, for $g > \dim H_1(M_1,\Q)$ hypotheses of Theorem \ref{thm:product} are satisfied and the result follows.
\end{proof}

\medskip 
Finally, we are in a position to prove Theorem \ref{thm:main} for the class $\mathcal C$. 

\if{} 
\begin{theorem}\label{thm:proofmain}
 Let \(A \) be a set in  $ \mathcal C =\bigcup_{n\geq 0}\mathcal C_n$, and let $M_1$ and $M_2$ be a pair of closed oriented connected manifolds of the same dimension. Then there exist closed oriented connected manifolds $N_1,N_2$ such that:

\begin{enumerate}
\item
$D(N_1,N_2)=A;$
\item
\(N_2\) is a finite product of manifolds of the form \(E_{j,g}\) or \(X_n\);
\item
$
D(M_1\times N_1,M_2\times N_2)
=
D(M_1,M_2)\cdot A.
$
\end{enumerate}
\end{theorem}

\begin{proof}
Since $A \in \mathcal C =\bigcup_{n\geq 0}\mathcal C_n$, then  $A \in \mathcal C_n$ for some $n$. We prove the theorem by induction on \(n\). Assume first that \(A\in\mathcal C_0\). Then \(A\) is an additive submonoid of \(\mathbb Z\). By Corollary \ref{cor:additive}, there exist closed oriented connected manifolds \(N_1\) and \(N_2\) such that
\[
D(N_1,N_2)=A.
\]
Moreover, by the construction in the proof of
Corollary~\ref{cor:additive}, \(N_2\) is of the form
\(E_{j,g}\times X_n\).
Let now \(M_1\) and \(M_2\) be closed oriented connected manifolds of the same dimension. Since the parameters \(g\) and \(n\) occurring in the construction of \(N_2\) can be chosen arbitrarily large, we may choose them sufficiently large with respect to \(M_1\) so that the hypotheses of Theorem~\ref{thm:product} are satisfied. Hence
\[
D(M_1\times N_1,M_2\times N_2)
=
D(M_1,M_2)\cdot D(N_1,N_2)
=
D(M_1,M_2)\cdot A.
\]
Thus the assertion holds for every \(A\in\mathcal C_0\).

Assume now that the assertion holds for every set in \(\mathcal C_n\), and let
\(A\in\mathcal C_{n+1}\).
By definition of \(\mathcal C_{n+1}\), either
$
A=A_1+A_2
$
or
$
A=A_1\cdot A_2,
$
where  \(A_1, A_2\in\mathcal C_n\).
We first treat the product case, so let
\[
A=A_1\cdot A_2,
\]
where \(A_1,A_2\in\mathcal C_n\).
By the induction hypothesis, there exist closed oriented connected manifolds
\[
(N_{1,1},N_{1,2})
\qquad\text{and}\qquad
(N_{2,1},N_{2,2})
\]
such that
\[
D(N_{1,1},N_{1,2})=A_1,
\qquad
D(N_{2,1},N_{2,2})=A_2,
\]
and each of \(N_{1,2}\) and \(N_{2,2}\) is a finite product of manifolds of the form $E_{j,g}$ or $X_n$.

 When constructing the second pair, we choose all parameters occurring in \(N_{2,2}\) sufficiently large with respect to \(N_{1,1}\). Write
\[
N_{2,2}=\prod_r Y_r,
\]
where each \(Y_r\) is of the form $E_{j,g}$ or $X_n$.
By choosing the parameters sufficiently large, we may assume that for every factor \(Y_r\), either
\[
\dim(Y_r)>\dim(N_{1,1}),
\]
or there exists \(p\geq 1\) such that
\[
\dim H_p(Y_r;\Q)>\dim H_p(N_{1,1};\Q).
\]
Hence, Hypothesis~\eqref{cond2} of Theorem~\ref{thm:product} is satisfied. Moreover, each factor \(Y_r\) is not dominated by any non-trivial product, by Lemmas~\ref{lem:Einotdomprod} and \ref{lem:Xnotdomprod}. Therefore Hypothesis~\eqref{cond1} of Theorem~\ref{thm:product} is also satisfied.
Thus Theorem~\ref{thm:product} applies:
\[
D(N_{1,1}\times N_{2,1},
  N_{1,2}\times N_{2,2})
=
D(N_{1,1},N_{1,2})
\cdot D(N_{2,1},N_{2,2})
=
A_1\cdot A_2
=
A.
\]
Moreover, \(N_{1,2}\times N_{2,2}\) is again a finite product of manifolds of the form \(E_{j,g}\times X_n\).

Let now \(M_1\) and \(M_2\) be closed oriented connected manifolds of the same dimension. Since the parameters occurring in the construction of \(N_{1,2}\) and \(N_{2,2}\) can be chosen arbitrarily large, we may choose them successively sufficiently large with respect to \(M_1\) so that the hypotheses of Theorem~\ref{thm:product} are satisfied for the pair
\[
\bigl(M_1,\,
N_{1,1}\times N_{2,1}\bigr).
\]
Applying Theorem~\ref{thm:product} once more, we obtain
\[
D\bigl(M_1\times N_{1,1}\times N_{2,1},
       M_2\times N_{1,2}\times N_{2,2}\bigr)
=
D(M_1,M_2)\cdot
D(N_{1,1}\times N_{2,1},
  N_{1,2}\times N_{2,2}).
\]
Since
\[
D(N_{1,1}\times N_{2,1},
  N_{1,2}\times N_{2,2})
=A,
\]
it follows that
\[
D\bigl(M_1\times N_{1,1}\times N_{2,1},
       M_2\times N_{1,2}\times N_{2,2}\bigr)
=
D(M_1,M_2)\cdot A.
\]

We  now turn to the sum case. Let
$
A=A_1+A_2,
$
where \(A_1,A_2\in\mathcal C_n\).
By the induction hypothesis, there exist closed oriented connected manifolds
\[
(N_{1,1},N_{1,2})
\qquad\text{and}\qquad
(N_{2,1},N_{2,2})
\]
such that
\[
D(N_{1,1},N_{1,2})=A_1,
\qquad
D(N_{2,1},N_{2,2})=A_2,
\]
and each \(N_{k,2}\) is a finite product of manifolds of the form
\(E_{j,g}\times X_n\).

For \(k\in\{1,2\}\), write
\[
N_{k,2}
=
\prod_{l=1}^{m_k}
(E_{j,g_{k,l}}\times X_{n_{k,l}}).
\]
Define
\[
N'_{k,1}
=
\prod_{l=1}^{m_k}
(E_{j,g_{k,l}}\times X_{n_{k,l}}^{\# s_{n_{k,l}}}).
\]
By construction,
\[
D(N'_{k,1},N_{k,2})=\mathbb Z_{\ge0}.
\]

Choosing the parameters in the second pair sufficiently large with respect to
\(N_{1,1}\), the hypotheses of Theorem~\ref{thm:product} are satisfied. Hence, for \(k\in\{1,2\}\),
\[
D(N_{k,1}\times N'_{3-k,1},
  N_{1,2}\times N_{2,2})
=
D(N_{k,1},N_{k,2})\cdot
D(N'_{3-k,1},N_{3-k,2})
=
A_k\cdot\mathbb Z_{\ge0}
=
A_k.
\]

Since \(N_{1,2}\times N_{2,2}\) is aspherical, Lemma~\ref{lem:sum} gives
\begin{multline*}
D\bigl(
(N_{1,1}\times N'_{2,1})
\#
(N_{2,1}\times N'_{1,1}),
\,N_{1,2}\times N_{2,2}
\bigr)
\\
=
D(N_{1,1}\times N'_{2,1},
  N_{1,2}\times N_{2,2})
+
D(N_{2,1}\times N'_{1,1},
  N_{1,2}\times N_{2,2})
=
A_1+A_2
=
A.
\end{multline*}

By Proposition \ref{prop:decomposition}, every additive submonoid of $\Z$ can be written as a finite sum of sets of the form
$
d\,\Z_{\geq 0},
\; d\in\Z.
$
Hence every set $A$ as in the statement can be obtained from sets of the form
$d\,\Z_{\geq 0}$ by finitely many sums and products.

We argue by induction on the complexity of such an expression for $A$.

First consider a building block
$
d_0\,\Z_{\geq 0}.
$
Let
$
j=\prod d,
$
where the product ranges over all integers $d$ associated with the building blocks occurring in the chosen expression of $A$, and let
$
i=\frac{j}{d_0}.
$
Then, for every $g,n>1$, Proposition \ref{cor:dZ>=0} gives
$$
D(E_{i,g}\times X_n^{\# s_n},E_{j,g}\times X_n)
=
d_0\,\Z_{\geq 0}.
$$
Thus each building block can be realized by a pair of the form
$$
(E_{i,g}\times X_n^{\# s_n},\, E_{j,g}\times X_n).
$$
As shown in the previous sections, $E_{j,g}$ and $X_n$ are aspherical, and hence so is their product. Furthermore, the parameters $g$ and $n$ may be chosen arbitrarily large.

We first treat the product operation. Assume inductively that $A_1$ and $A_2$ have been realized by pairs
$$
(N_{1,1},N_{1,2})
\qquad\text{and}\qquad
(N_{2,1},N_{2,2}),
$$
satisfying
$$
D(N_{1,1},N_{1,2})=A_1,
\qquad
D(N_{2,1},N_{2,2})=A_2,
$$
and such that each $N_{k,2}$, $k=1, 2$  is a finite product of manifolds of the form
$E_{j,g}$ and $X_n$.

When constructing the second pair, we choose all parameters $g$ and $n$ sufficiently large with respect to $N_{1,1}$. More precisely, writing
$
N_{2,2}=\prod_r Y_r,
$
where each $Y_r$ is either of the form $E_{j,g}$ or $X_n$, we choose them so that, for every factor $Y_r$, either
$
\dim(Y_r)>\dim(N_{1,1})
$
or there exists $p\geq 1$ such that
$
\dim H_p(Y_r,\Q)>\dim H_p(N_{1,1},\Q).
$
It follows that, for every map
$
f\colon N_{1,1}\to Y_r,
$
the induced homomorphism
$
H_p(f)\colon H_p(N_{1,1},\Q)\to H_p(Y_r,\Q)
$
fails to be surjective for some $p\geq 1$. Thus Hypothesis~\eqref{cond2} of Theorem~\ref{thm:product} is satisfied with $M_1=N_{1,1}$.

Moreover, each factor $Y_r$ is not dominated by any nontrivial product, by Lemmas \ref{lem:Einotdomprod} and \ref{lem:Xnotdomprod}. Hence Hypothesis~\eqref{cond1} of Theorem \ref{thm:product} is also satisfied.

Therefore Theorem \ref{thm:product} applies and yields
$$
D(N_{1,1}\times N_{2,1},
  N_{1,2}\times N_{2,2})
=
D(N_{1,1},N_{1,2})
D(N_{2,1},N_{2,2})
=
A_1\cdot A_2.
$$

We now turn to the sum operation.
        For $k\in\{1,2\}$, $N_{k,2}=\prod_{l=1}^{m_k} (E_{j,g_{k,l}}\times X_{n_{k,l}})$ by hypothesis. Define the manifold $N'_{k,1}=\prod_{l=1}^{m_k} (E_{j,g_{k,l}}\times X^{\#s_{n_{k,l}}}_{n_{k,l}})$ for $k\in\{1,2\}$. It is clear that $\dim(N_{k,2})=\dim(N'_{k,1})$. Suppose by construction that $D(N'_{k,1},N_{k,2})=\Z_{\geq 0}$ and that $\dim(H_1(N_{k,1},\Q))\geq \dim(H_1(N'_{k,1},\Q))$. Then, $\dim(H_1(N_{1,1}\times N_{2,1},\Q))\geq \dim(H_1(N'_{1,1}\times N'_{2,1},\Q))$ and by Theorem \ref{thm:product}
    $$D(N'_{1,1}\times N'_{2,1},N_{1,2}\times N_{2,2})=D(N'_{1,1},N_{1,2}) D(N'_{2,1},N_{2,2})=\Z_{\geq 0}\cdot \Z_{\geq 0}=\Z_{\geq 0}.$$

    Notice that every set $A$ that is an iterated sum and product of sets $d\,\Z_{\geq 0}$ satisfies $A\cdot \Z_{\geq 0}=A$. Therefore, applying Theorem \ref{thm:product},
    $$D(N_{k,1}\times N'_{3-k,1},N_{1,2}\times N_{2,2})=D(N_{k,1},N_{k,2})D(N'_{3-k,1},N_{3-k,2})=A_k\cdot \Z_{\geq0}=A_k$$
    for $k\in\{1,2\}$ and, since $N_{1,2}\times N_{2,2}$ is aspherical, by Lemma \ref{lem:sum},
    \begin{multline*}
        D((N_{1,1}\times N'_{2,1})\#(N_{2,1}\times N'_{1,1}),N_{1,2}\times N_{2,2})=\\=D(N_{1,1}\times N'_{2,1},N_{1,2}\times N_{2,2})+D(N_{2,1}\times N'_{1,2},N_{1,1}\times N_{2,2})=A_1+A_2.
    \end{multline*}

    Finally, one checks that $\dim(H_1((N_{1,1}\times N'_{2,1})\#(N_{2,1}\times N'_{1,1}),\Q))\geq \dim(H_1(N'_{1,1}\times N'_{2,1},\Q))$. 
 \end{proof}
 \fi 

\begin{theorem}\label{thm:proofmain}
 Let \(A \) be a set in  $ \mathcal C =\bigcup_{n\geq 0}\mathcal C_n$, and let $M_1$ and $M_2$ be a pair of closed oriented connected manifolds of the same dimension. Then there exist closed oriented connected manifolds $N_1,N'_1,N_2$ such that:

\begin{enumerate}
\item\label{property1}
$D(N_1,N_2)=A;$
\item\label{property2}
$D(N'_1,N_2)=\Z_{\geq 0};$
\item
\(N_2\) is a finite product of manifolds of the form \(E_{j,g}\) or \(X_n\), with  parameters $g$ and $n$ arbitrarily large, $n\not\equiv 1 \pmod 3$;
\item
$
D(M_1\times N_1,M_2\times N_2)
=
D(M_1,M_2)\cdot A.
$
\end{enumerate}
\end{theorem}

\begin{proof}
Since $A \in \mathcal C =\bigcup_{n\geq 0}\mathcal C_n$, then $A \in \mathcal C_n$ for some $n$. We prove the theorem by induction on $n$. Assume first that $A\in\mathcal C_0$. Then $A$ is an additive submonoid of $\mathbb Z$. By Corollary~\ref{cor:additive}, there exist closed oriented connected manifolds \(N_1\) and \(N_2\) such that
$$D(N_1,N_2)=A.$$
Moreover, by the construction in the proof of
Corollary~\ref{cor:additive}, \(N_2\) is of the form
\(E_{j,g}\times X_n\). Letting \(N'_1=E_{j,g}\times X_n^{\#s_n}\), \(D(N'_1,N_2)=\Z_{\geq 0}\) by Corollary~\ref{cor:dZ}. 

Let now \(M_1\) and \(M_2\) be closed oriented connected manifolds of the same dimension. Since the parameters \(g\) and \(n\) occurring in the construction of \(N_2\) can be chosen arbitrarily large, we may choose them sufficiently large with respect to \(M_1\) so that the hypotheses of Theorem~\ref{thm:product} are satisfied. Hence
\[
D(M_1\times N_1,M_2\times N_2)
=
D(M_1,M_2)\cdot D(N_1,N_2)
=
D(M_1,M_2)\cdot A.
\]
Thus the assertion holds for every \(A\in\mathcal C_0\).

Assume now that the assertion holds for every set in
\(\bigcup_{j\leq n}\mathcal C_j\), and let \(A\in\mathcal C_{n+1}\). By definition, either
\[
A=A_1+A_2
\qquad\text{or}\qquad
A=A_1\cdot A_2,
\]
where \(A_1,A_2\in\bigcup_{j\leq n}\mathcal C_j\).
By the induction hypothesis, there exist closed oriented connected manifolds
\[
(N_{1,1},N'_{1,1},N_{1,2})
\qquad\text{and}\qquad
(N_{2,1},N'_{2,1},N_{2,2})
\]
such that
\[
D(N_{1,1},N_{1,2})=A_1,
\qquad
D(N_{2,1},N_{2,2})=A_2,
\]
\[
D(N'_{1,1},N_{1,2})=D(N'_{2,1},N_{2,2})=\Z_{\geq 0},
\]
and each of \(N_{1,2}\) and \(N_{2,2}\) is a finite product of manifolds of the form $E_{j,g}$ or $X_n$.

 When constructing the second pair, we choose all parameters $g$ and $n$ occurring in \(N_{2,2}\) sufficiently large with respect to \(N_{1,1}\) and \(N'_{1,1}\). Write
\[
N_{2,2}=\prod_r Y_r,
\]
where each \(Y_r\) is of the form $E_{j,g}$ or $X_n$.
By choosing the parameters sufficiently large, we may assume that for every factor \(Y_r\), either
\[
\dim(Y_r)>\dim(N_{1,1})= \dim(N'_{1,1}),
\]
or 
\[
\dim H_1(Y_r;\Q)>\dim H_1(N_{1,1};\Q),\dim H_1(N'_{1,1};\Q).
\]
Hence, Hypothesis~\eqref{cond2} of Theorem~\ref{thm:product} is satisfied. Moreover, each factor \(Y_r\) is not dominated by any non-trivial product, by Lemmas~\ref{lem:Einotdomprod} and \ref{lem:Xnotdomprod}. Therefore Hypothesis~\eqref{cond1} of Theorem~\ref{thm:product} is also satisfied, and we obtain the following:
\begin{align*}
    D(N_{1,1}\times N_{2,1},
    N_{1,2}\times N_{2,2})
    &=
    D(N_{1,1},N_{1,2})
    \cdot D(N_{2,1},N_{2,2})\\
    & =
    A_1\cdot A_2,\\
{}
     D(N'_{1,1}\times N'_{2,1},
    N_{1,2}\times N_{2,2})
    & =D(N'_{1,1},N_{1,2})
    \cdot D(N'_{2,1},N_{2,2})\\
    & =
    \Z_{\geq 0}\cdot \Z_{\geq 0}=\Z_{\geq 0},   \\
    D(N_{1,1}\times N'_{2,1},
{}    
  N_{1,2}\times N_{2,2}) 
  &= D(N_{1,1},N_{1,2})\cdot
D(N'_{2,1},N_{2,2})\\
 & =
A_1\cdot\mathbb Z_{\ge0}
=
A_1,\\
    D(N'_{1,1}\times N_{2,1},
  N_{1,2}\times N_{2,2})
 & =
D(N'_{1,1},N_{1,2})\cdot D(N_{2,1},N_{2,2})\\ 
& =
\mathbb Z_{\ge0}\cdot A_2
=
A_2.
\end{align*}
Then, for \(N_1':= N_{1,1}'\times N_{2,1}'\) and \(N_2:= N_{1,2}\times N_{2,2}\), the second equation is property~\eqref{property2}, and \(N_2\) is again a finite product of manifolds of the form \(E_{j,g}\) or \(X_n\) with parameters $g$ and $n$ arbitrarily large, $n\not\equiv 1 \pmod 3$. Moreover, for \(N_1:= N_{1,1}\times N_{2,1}\), the first equation is the product case
$
A=A_1\cdot A_2
$
of property~\eqref{property1}.

We  now turn to the sum case
$
A=A_1+A_2
$.
Since \(N_{1,2}\times N_{2,2}\) is aspherical, Lemma~\ref{lem:sum} gives
\begin{multline*}
D\bigl(
(N_{1,1}\times N'_{2,1})
\#
(N'_{1,1}\times N_{2,1}),
\,N_{1,2}\times N_{2,2}
\bigr)=
\\
=
D(N_{1,1}\times N'_{2,1},
  N_{1,2}\times N_{2,2})
+
D(N'_{1,1}\times N_{2,1},
  N_{1,2}\times N_{2,2})=
\\
=
A_1+A_2
=
A.
\end{multline*}
Thus, we write $N_1= N_{1,1}\times N_{2,1}$ in the product case.

Let now \(M_1\) and \(M_2\) be closed oriented connected manifolds of the same dimension. 
Since the parameters occurring in the construction of \(N_{1,2}\) and \(N_{2,2}\) can be chosen arbitrarily large, we may choose them successively sufficiently large with respect to \(M_1\) so that the hypotheses of Theorem~\ref{thm:product} are satisfied for the pair
\[
\bigl(M_1 \times N_1,\,
M_2 \times N_2\bigr).
\]
Applying Theorem~\ref{thm:product} once more, we obtain
\[
D(M_1\times N_1,
       M_2\times N_2)
=
D(M_1,M_2)\cdot
D(N_1,
  N_2).
\]
Since
\[
D(N_1,
  N_2)
=A,
\]
it follows that
\[
D(M_1\times N_1,
       M_2\times N_2)
=
D(M_1,M_2)\cdot A.
\]

\if{}
By Proposition \ref{prop:decomposition}, every additive submonoid of $\Z$ can be written as a finite sum of sets of the form
$
d\,\Z_{\geq 0},
\; d\in\Z.
$
Hence every set $A$ as in the statement can be obtained from sets of the form
$d\,\Z_{\geq 0}$ by finitely many sums and products.

We argue by induction on the complexity of such an expression for $A$.

First consider a building block
$
d_0\,\Z_{\geq 0}.
$
Let
$
j=\prod d,
$
where the product ranges over all integers $d$ associated with the building blocks occurring in the chosen expression of $A$, and let
$
i=\frac{j}{d_0}.
$
Then, for every $g,n>1$, Proposition \ref{cor:dZ>=0} gives
$$
D(E_{i,g}\times X_n^{\# s_n},E_{j,g}\times X_n)
=
d_0\,\Z_{\geq 0}.
$$
Thus each building block can be realized by a pair of the form
$$
(E_{i,g}\times X_n^{\# s_n},\, E_{j,g}\times X_n).
$$
As shown in the previous sections, $E_{j,g}$ and $X_n$ are aspherical, and hence so is their product. Furthermore, the parameters $g$ and $n$ may be chosen arbitrarily large.

We first treat the product operation. Assume inductively that $A_1$ and $A_2$ have been realized by pairs
$$
(N_{1,1},N_{1,2})
\qquad\text{and}\qquad
(N_{2,1},N_{2,2}),
$$
satisfying
$$
D(N_{1,1},N_{1,2})=A_1,
\qquad
D(N_{2,1},N_{2,2})=A_2,
$$
and such that each $N_{k,2}$, $k=1, 2$  is a finite product of manifolds of the form
$E_{j,g}$ and $X_n$.

When constructing the second pair, we choose all parameters $g$ and $n$ sufficiently large with respect to $N_{1,1}$. More precisely, writing
$
N_{2,2}=\prod_r Y_r,
$
where each $Y_r$ is either of the form $E_{j,g}$ or $X_n$, we choose them so that, for every factor $Y_r$, either
$
\dim(Y_r)>\dim(N_{1,1})
$
or there exists $p\geq 1$ such that
$
\dim H_p(Y_r,\Q)>\dim H_p(N_{1,1},\Q).
$
It follows that, for every map
$
f\colon N_{1,1}\to Y_r,
$
the induced homomorphism
$
H_p(f)\colon H_p(N_{1,1},\Q)\to H_p(Y_r,\Q)
$
fails to be surjective for some $p\geq 1$. Thus Hypothesis~\eqref{cond2} of Theorem \ref{thm:product} is satisfied with $M_1=N_{1,1}$.

Moreover, each factor $Y_r$ is not dominated by any nontrivial product, by Lemmas \ref{lem:Einotdomprod} and \ref{lem:Xnotdomprod}. Hence Hypothesis~\eqref{cond1} of Theorem \ref{thm:product} is also satisfied.

Therefore Theorem \ref{thm:product} applies and yields
$$
D(N_{1,1}\times N_{2,1},
  N_{1,2}\times N_{2,2})
=
D(N_{1,1},N_{1,2})
D(N_{2,1},N_{2,2})
=
A_1\cdot A_2.
$$

We now turn to the sum operation.
        For $k\in\{1,2\}$, $N_{k,2}=\prod_{l=1}^{m_k} (E_{j,g_{k,l}}\times X_{n_{k,l}})$ by hypothesis. Define the manifold $N'_{k,1}=\prod_{l=1}^{m_k} (E_{j,g_{k,l}}\times X^{\#s_{n_{k,l}}}_{n_{k,l}})$ for $k\in\{1,2\}$. It is clear that $\dim(N_{k,2})=\dim(N'_{k,1})$. Suppose by construction that $D(N'_{k,1},N_{k,2})=\Z_{\geq 0}$ and that $\dim(H_1(N_{k,1},\Q))\geq \dim(H_1(N'_{k,1},\Q))$. Then, $\dim(H_1(N_{1,1}\times N_{2,1},\Q))\geq \dim(H_1(N'_{1,1}\times N'_{2,1},\Q))$ and by Theorem \ref{thm:product}
    $$D(N'_{1,1}\times N'_{2,1},N_{1,2}\times N_{2,2})=D(N'_{1,1},N_{1,2}) D(N'_{2,1},N_{2,2})=\Z_{\geq 0}\cdot \Z_{\geq 0}=\Z_{\geq 0}.$$

    Notice that every set $A$ that is an iterated sum and product of sets $d\,\Z_{\geq 0}$ satisfies $A\cdot \Z_{\geq 0}=A$. Therefore, applying Theorem \ref{thm:product},
    $$D(N_{k,1}\times N'_{3-k,1},N_{1,2}\times N_{2,2})=D(N_{k,1},N_{k,2})D(N'_{3-k,1},N_{3-k,2})=A_k\cdot \Z_{\geq0}=A_k$$
    for $k\in\{1,2\}$ and, since $N_{1,2}\times N_{2,2}$ is aspherical, by Lemma \ref{lem:sum},
    \begin{multline*}
        D((N_{1,1}\times N'_{2,1})\#(N_{2,1}\times N'_{1,1}),N_{1,2}\times N_{2,2})=\\=D(N_{1,1}\times N'_{2,1},N_{1,2}\times N_{2,2})+D(N_{2,1}\times N'_{1,2},N_{1,1}\times N_{2,2})=A_1+A_2.
    \end{multline*}

    Finally, one checks that $\dim(H_1((N_{1,1}\times N'_{2,1})\#(N_{2,1}\times N'_{1,1}),\Q))\geq \dim(H_1(N'_{1,1}\times N'_{2,1},\Q))$.
\fi 
 
\end{proof}

\begin{remark}
    By construction, a simple verification shows that $$\dim H_1(N_1;\Q)\geq\dim H_1(N'_1;\Q),$$ so in the proof of Theorem \ref{thm:proofmain} it was sufficient to require the parameters $g$ and $n$ occurring in \(N_{2,2}\) to be sufficiently large with respect to \(N_{1,1}\) and \(M_1\).
\end{remark}

\bibliographystyle{abbrv}
\bibliography{refCMuV3}
\end{document}